\documentclass[dvipdfmx,11pt]{amsart}
\usepackage{url}
\usepackage[dvipdfmx,colorlinks=true]{hyperref}
\numberwithin{equation}{section}
\theoremstyle{definition}
\newtheorem{thm}{Theorem}[section]
\newtheorem{prop}[thm]{Proposition}
\newtheorem{df}[thm]{Definition}
\newtheorem{lem}[thm]{Lemma}
\newtheorem{rem}[thm]{Remark}

\newtheorem{ex}[thm]{Example}

\newtheorem*{ack}{Acknowledgments}
\newtheorem*{mt}{Main Theorem}
\def\Aut{\mathop{\mathrm{Aut}}\nolimits}

\def\Hom{\mathop{\mathrm{Hom}}\nolimits}
\def\Km{\mathop{\mathrm{Km}}\nolimits}
\title[Galois points and Eisenstein $K3$ surfaces]{Quartic surfaces with a Galois point and Eisenstein $K3$ surfaces}
\author[K.~Miura]{Kei Miura}
\address{Department of Mathematics, National Institute of Technology, Ube College, 
2-14-1 Tokiwadai, Ube, Yamaguchi 755- 8555, Japan}
\email{kmiura@ube-k.ac.jp}
\author[S.~Taki]{Shingo Taki}
\address{Department of Mathematics, Tokai University,
4-1-1, Kitakaname, Hiratsuka, Kanagawa, 259-1292, Japan}
\email{taki@tsc.u-tokai.ac.jp}
\urladdr{http://www2.sm.u-tokai.ac.jp/taki/}
\date{\today}
\subjclass[2020]{Primary 14J70; Secondary 14J28, 14J50, 14N05}
\keywords{Galois point, automorphism, $K3$ surface}
\dedicatory{}
\thanks{}

\begin{document}

\begin{abstract}
We prove that there exists a one to one correspondence between
smooth quartic surfaces with an inner Galois point
and  Eisenstein $K3$ surfaces of type $(4, 3)$.
Furthermore we characterize 
the quartic surface with 8 (the maximum number) inner Galois points 
as a singular $K3$ surface.
\end{abstract}

\maketitle


\section{Introduction}\label{Introduction}
In algebraic geometry, hypersurfaces of a projective space 
are one of the most fundamental and important objects.
Yoshihara introduced a remarkable new notion, that is, a Galois point 
in 1996.
The paper in which a Galois point was first introduced is \cite{MY}. 
In that paper, they constructed a field theory for
function fields of smooth plane quartic curves
and introduced Galois points as a geometric viewpoint for field extensions.
After that, Galois points have been studied by several mathematicians,
and generalized to various types, 
e.g., \cite{Yoshihara1, M, Yoshihara2, Homma, Fukasawa} and so on.

Yoshihara completely determined quartic surfaces which contain Galois points \cite{Yoshihara-GK3}.
In particular he proved that the maximum number of Galois points on quartic surfaces is 8.
Such a Galois point naturally induces an automorphism on a quartic $K3$ surface.
It is known that there exist many results about  automorphisms of $K3$ surfaces, 
e.g., \cite{Ni, AS, Taki1} and so on.
This paper is devoted to a study of a correspondence 
between quartic surfaces which contain a Galois point and 
$K3$ surfaces with an automorphism.
We will work over $\mathbb{C}$, the field of complex numbers, throughout this paper.

Let $V$ be a smooth hypersurface in $\mathbb{P}^{N}$
and $\mathbb{C}(V)$ the function field of $V$.
For a point $P \in V$, we consider a 
projection $\pi_{P}: V \dashrightarrow H$ where 
$H$ is a hyperplane in $\mathbb{P}^{N}$ which does not contain $P$.
Note that the projection $\pi_{P}$ induces the extension 
$\mathbb{C}(V)/\pi_{P}^{\ast}\mathbb{C}(H)$ of function fields.
The point $P$ is called an \textit{inner Galois point} for $V$ if 
the extension $\mathbb{C}(V)/\pi_{P}^{\ast}\mathbb{C}(H)$ is Galois. 
Moreover we denote by $G_{P}$ its Galois group.

Let $S$ be a $K3$ surface with an automorphism group 
$G \subset \Aut (S)$ of order 3
which acts on $H^{0}(S, \Omega_{S}^{2})$ faithfully, 
and $L:=H^{2}(S, \mathbb{Z})^{G}$ the $G$-invariant sublattice of $H^{2}(S, \mathbb{Z})$. 
Then $L$ is a 3-elementary lattice, namely 
$\Hom (L, \mathbb{Z})/L \simeq (\mathbb{Z}/3\mathbb{Z})^{a}$.
If the rank of $L$ is $r$ then a pair $(S, G)$ is called an 
\textit{Eisenstein $K3$ surface} of type $(r, a)$.
Generators of $G$ are classified by \cite{AS} and \cite{Taki1},
and several applications (e.g., \cite{OT} and \cite{MOT}) are known.

The following is the main theorem of this paper.

\begin{mt}
\begin{enumerate}
\item\label{1} Let $S$ be a smooth quartic surface with an inner Galois point $P$
and $G_{P}$ the Galois group associated with $P$.
Then a pair $(S, G_{P})$ is an Eisenstein $K3$ surface of type $(4, 3)$.
In particular  there exists a one to one correspondence between
smooth quartic surfaces with an inner Galois point
and  Eisenstein $K3$ surfaces of type $(4, 3)$.

\item\label{2} Let $S_{8}$ be the quartic surface with the maximal number of inner Galois points.
Then $S_{8}$ is the singular $K3$ surface whose the transcendental lattice is 
given by $\begin{pmatrix}8 & 4\\ 4 & 8\end{pmatrix}$.
\end{enumerate}
\end{mt}
\begin{rem}
Main Theorem (\ref{2}) solves an open problem for Galois points.
See also \cite[\fbox{I}-(B)-(1)-(b)]{open-problems}.
\end{rem}

We summarize the contents of this paper. 
In Section \ref{Eisenstein_K3}, we recall some basic results about 
Eisenstein $K3$ surfaces. 
In Section \ref{QSwGisEK}, we discuss Main Theorem (\ref{1}).
The proof is given by Theorem \ref{mt1} and Proposition \ref{1to1}.
In Section \ref{s8hax3}, we treat Main Theorem (\ref{2}).
The proof is given by Theorem \ref{mt2}. 

\begin{ack}
We would like to thank Professor Kazuki Utsumi for several useful comments.
Special thanks go to Professor Masato Kuwata for pointing out a mistake in an earlier version.
This work was partially supported by 
Grant-in-Aid for Scientific Research (C) 18K03230, 19K03454 and 23K03036 from JSPS.
\end{ack}

\section{Eisenstein $K3$ surfaces}\label{Eisenstein_K3}
In this section, we collect some basic results for Eisenstein $K3$ surfaces. 
For the details about foundations $K3$ surfaces, see \cite{Huybrechts, Kondo-book}, and so on.

Let $S$ be a compact complex surface. 
If its canonical line bundle $K_{S}$ is trivial and $H^{1}(S, \mathcal{O}_{S})=0$ then 
$S$ is called a \textit{$K3$ surface}.
We consider the cup product on $H^{2}(S,\mathbb{Z})$:
\[ \langle \ , \ \rangle :H^{2}(S,\mathbb{Z})\times H^{2}(S,\mathbb{Z})\to \mathbb{Z}.\]
Then the pair $(H^{2}(S,\mathbb{Z}), \langle \ , \ \rangle )$
has a structure of a lattice.
Here a lattice  means a free abelian group $L$ of finite rank equipped with 
a non-degenerate symmetric bilinear form, which will be denote by $\langle \ , \ \rangle $.
The bilinear form $\langle \ , \ \rangle $ determines a 
canonical embedding $L\subset L^{\ast }=\Hom (L,\mathbb{Z})$. 

Let $G$ be an automorphism group $G \subset \Aut (S)$ which 
acts on $H^{0}(S, \Omega_{S}^{2})$ faithfully, and 
$L:=\{ x\in H^{2}(S, \mathbb{Z})\mid g^{\ast}(x)=x, \forall g \in G \}$ 
the $G$-invariant sublattice of $H^{2}(S, \mathbb{Z})$. 
It is known that if the order of $G$ is prime $p$ then
$L^{\ast}/L $ is isomorphic to a $p$-elementary abelian group.

\begin{df}(\cite[\S 3]{MOT})\label{df-Eisenstein}
Let $S$ be a $K3$ surface with an automorphism group 
$G$ of order 3 which acts on $H^{0}(S, \Omega_{S}^{2})$ faithfully.
If $G$-invariant sublattice $L$ is of rank $r$ and satisfies
$L^{\ast}/L \simeq (\mathbb{Z}/3\mathbb{Z})^{a}$
then a pair $(S, G)$ is called an \textit{Eisenstein $K3$ surface} of type $(r, a)$.
\end{df}

\begin{prop}(\cite{AS}, \cite{Taki1})\label{order3-basic}
The followings hold.
\begin{enumerate}
\item Let $(S, G)$ be an Eisenstein $K3$ surface.
The fixed locus $S^{G}$ is of the form
\[ S^{G}=C^{(g)} \amalg E_{1} \amalg \dots \amalg E_{k} \amalg\{P_{1},\dots , P_{n}\}, \]
where $C^{(g)}$ is a genus $g$ curve, $E_{i}$ are non-singular rational curves 
and $P_{j}$ are isolated points with
\[ g=\frac{22-r-2a}{4}, \hspace{5mm} k=\frac{2+r-2a}{4}, \hspace{5mm} n=\frac{r}{2}-1.\]

\item The deformation type of an Eisenstein $K3$ surface $(S, G)$ is determined by the invariant $(r, a)$.
\end{enumerate}
\end{prop}

\begin{rem}\label{symplectic_order3}
Let $S$ be a $K3$ surface with an automorphism group $G$ of order 3.
Assume that $G$ does not act on $H^{0}(S, \Omega_{S}^{2})=\langle \omega_{S} \rangle$ faithfully,
hence there exists $g\neq 1 \in G$ such that $g^{\ast}\omega_{S}=\omega_{S}$.
Then the fixed locus $S^{G}$ consists of exactly 6 isolated points,
hence $S^{G}$ does not contain a curve. See also \cite[\S 5]{Ni}.
\end{rem}

\begin{ex}\label{ex-4-3}
Let $S$ be a smooth quartic surface given by the homogeneous equation:
$F_{4}(X, Y, Z)+F_{1}(X, Y, Z)W^{3}=0$ where 
$F_{d}(X, Y, Z)$ is a homogeneous polynomial of degree $d$.
Then $S$ has an automorphism 
$\sigma: [X : Y : Z : W] \mapsto [X : Y : Z : \zeta_{3}W]$
where $\zeta_{n}$ is a primitive $n$-th root of unity.

Let $G$ be an automorphism group generated by $\sigma$.
Then we have
\begin{align*}
S^{G}&=S \cap (\{W=0\}\amalg \{X=Y=Z=0\})\\
&=\{F_{4}(X, Y, Z)=0\}\amalg \{[0:0:0:1]\}\\
&=C^{(3)}\amalg \{[0:0:0:1]\}.
\end{align*}

It is easy to see that $(r, a)=(4, 3)$ by Proposition \ref{order3-basic}, 
hence a pair $(S, G)$ is Eisenstein $K3$ surface of type $(4, 3)$.
\end{ex}

\begin{rem}\label{rem_AS49}
Let $(S, G)$ be an Eisenstein $K3$ surface of type $(4, 3)$.
Then $S$ is isomorphic to a smooth quartic surface 
with equations of the form $F_{4}(X, Y, Z)+F_{1}(X, Y, Z)W^{3}=0$
and $G$ is generated by the automorphism 
$\sigma : [X : Y : Z : W] \mapsto [X : Y : Z : \zeta_{3}W]$.
See also \cite[Proposition 4.9]{AS}.
\end{rem}

\section{Quartic surfaces with inner Galois points}\label{QSwGisEK}
In this section, we treat quartic surfaces with an inner Galois point.
For the details, see \cite{KTY, Yoshihara-GK3}, and so on.

Again, we give the definition of Galois points.

\begin{df}
Let $V$ be a smooth hypersurface in $\mathbb{P}^{N}$
and $\mathbb{C}(V)$ the function field of $V$.
For a point $P \in \mathbb{P}^{N}$, we consider a 
projection $\pi_{P}: V \dashrightarrow H$ where 
$H$ is a hyperplane in $\mathbb{P}^{N}$ which does not contain $P$.
The point $P \in V$ (respectively $P\not \in V$) is called 
an \textit{inner (resp. outer) Galois point} for $V$ if 
the extension $\mathbb{C}(V)/\pi_{P}^{\ast}\mathbb{C}(H)$ of function fields,
induced from $\pi_{P}$, is Galois. 
Moreover we denote by $G_{P}$ the Galois group of $\mathbb{C}(V)/\pi_{P}^{\ast}\mathbb{C}(H)$.
\end{df}

\begin{rem}
The structure of the extension $\mathbb{C}(V)/\pi_{P}^{\ast}\mathbb{C}(H)$ 
does not depend on choice of $H$ but only on $P$.
\end{rem}

\begin{prop}\label{equation_place}(\cite{Yoshihara-GK3}) 
Let $S$ be a smooth quartic surface with an inner Galois point.
Then $S$ is given by one of the followings (up to projective transformations): 
\begin{enumerate}
\item $F_{4}(X, Y, Z)+F_{1}(X, Y, Z)W^{3}=0$,
\item $F_{4}(X, Z)+F_{1}(X, Z)(Y^{3}+W^{3})=0$,
\item $F_{4}(X, Y)+Z^{4}+ZW^{3}=0$,
\item $X^{4}+Z^{4}+XY^{3}+ZW^{3}=0$,
\end{enumerate}
where $F_{d}$ is a homogeneous polynomial of degree $d$ with distinct factors.
Moreover the Galois points on each $S$ are
\begin{enumerate}
\item $[0:0:0:1]$,
\item $[0:1:0:0], [0:0:0:1]$,
\item $[0:0:0:1], [0:0: \zeta_{6}^{i}:1] \ \ (i=1, 3, 5)$,
\item $[0:0:0:1], [0:0: \zeta_{6}^{i}:1], [0:1:0:0], [\zeta_{6}^{i}:1:0:0 ] \ \ (i=1, 3, 5)$,
\end{enumerate}
respectively.
\end{prop}

\begin{lem}\label{generator}(\cite[Lemma 2]{KTY})
Put 
$P_{1}:=[0:0:0:1], 
P_{2}:=[0:0: \zeta_{6}:1], 
P_{3}:=[0:0: \zeta_{6}^{3}:1], 
P_{4}:=[0:0: \zeta_{6}^{5}:1],
P_{5}=[0:1:0:0], 
P_{6}:=[\zeta_{6}:1:0:0 ], 
P_{7}:=[\zeta_{6}^{3}:1:0:0 ]$ and
$P_{8}:=[\zeta_{6}^{5}:1:0:0 ]$.
Let $\sigma_{i}$ be a generator of $G_{P_{i}} \ \ (i=1, 2, \dots , 8)$.
Then $\sigma_{i}$ is given by one of the followings:
\begin{align*}
\sigma_{1} &: [X:Y:Z:W] \mapsto [X:Y:Z:\zeta_{3}W], \\
\sigma_{2} &: [X:Y:Z:W] \mapsto \left[ X:Y:
\frac{2\zeta_{6}-1}{3}Z+\frac{-\zeta_{6}-1}{3}W : \frac{4\zeta_{6}-2}{3}Z+\frac{\zeta_{6}+1}{3}W \right], \\
\sigma_{3} &: [X:Y:Z:W] \mapsto \left[ X:Y:
\frac{2\zeta_{6}-1}{3}Z+\frac{-\zeta_{6}+2}{3}W : \frac{-2\zeta_{6}+4}{3}Z+\frac{\zeta_{6}+1}{3}W \right], \\
\sigma_{4} &: [X:Y:Z:W] \mapsto \left[ X:Y:
\frac{2\zeta_{6}-1}{3}Z+\frac{2\zeta_{6}-1}{3}W : \frac{-2\zeta_{6}-2}{3}Z+\frac{\zeta_{6}+1}{3}W \right], \\
\sigma_{5} &: [X:Y:Z:W] \mapsto [X:\zeta_{3}Y:Z:W], \\
\sigma_{6} &: [X:Y:Z:W] \mapsto 
\left[ \frac{2\zeta_{6}-1}{3}X+\frac{-\zeta_{6}-1}{3}Y : \frac{4\zeta_{6}-2}{3}X+\frac{\zeta_{6}+1}{3}Y :Z:W \right], \\
\sigma_{7} &: [X:Y:Z:W] \mapsto 
\left[ \frac{2\zeta_{6}-1}{3}X+\frac{-\zeta_{6}+2}{3}Y : \frac{-2\zeta_{6}+4}{3}X+\frac{\zeta_{6}+1}{3}Y :Z:W \right], \\
\sigma_{8} &: [X:Y:Z:W] \mapsto 
\left[ \frac{2\zeta_{6}-1}{3}X+\frac{2\zeta_{6}-1}{3}Y : \frac{-2\zeta_{6}-2}{3}X+\frac{\zeta_{6}+1}{3}Y :Z:W \right].\\
\end{align*}
In particular, each $\sigma_{i}$ is of order 3.
\end{lem}

\begin{thm}\label{mt1}
Let $S$ be a smooth quartic surface with an inner Galois point $P$.
Then a pair $(S, G_{P})$ is an Eisenstein $K3$ surface of type $(4,3)$.
\end{thm}

\begin{proof}
We check for each Galois point. 
If $P=[0:0:0:1]$ then we may assume that 
$S$ is given by $F_{4}(X, Y, Z)+F_{1}(X, Y, Z)W^{3}=0$ and 
the generator of $G_{P}$ is given by 
$\sigma: [X:Y:Z:W] \mapsto [X:Y:Z:\zeta_{3}W]$
by Lemma \ref{generator}.
We remark that $G_{P}$ is of order 3.
Since fixed locus $S^{G_{P}}$ consists of one non-singular curve of genus 3 
and one isolated point by Example \ref{ex-4-3}, $G_{P}$ acts on 
$H^{0}(S, \Omega_{S}^{2})$ faithfully. See also Remark \ref{symplectic_order3}.
Thus, $(S, G_{P})$ is an Eisenstein $K3$ surface of type $(4,3)$.

If $P=[0:0: \zeta_{6}:1]$ then we may assume that 
$S$ is given by $F_{4}(X, Y)+Z^{4}+ZW^{3}=0$ and 
the generator of $G_{P}$ is given by 
$\sigma_{2}$ in Lemma \ref{generator}.
If $[X: Y: Z: W] \in S^{G_{P}} \subset \mathbb{P}^{3}$ then there exists a non-zero constant 
$k \in \mathbb{C}^{\ast}$ such that
\[ \begin{pmatrix}
1 & 0 & 0 & 0 \\
0 & 1 & 0 & 0 \\
0 & 0 & \frac{2\zeta_{6}-1}{3} & \frac{-\zeta_{6}-1}{3} \\
0 & 0 & \frac{4\zeta_{6}-2}{3} & \frac{\zeta_{6}+1}{3}
\end{pmatrix}
\begin{pmatrix} X \\ Y \\ Z \\ W \end{pmatrix}
=k \begin{pmatrix} X \\ Y \\ Z \\ W \end{pmatrix}.
\]
Since the system of equations has a nontrivial solution
$(X, Y, Z, W)\neq (0, 0, 0, 0)$, we have 
\begin{align*}
\begin{vmatrix}
1-k & 0 & 0 & 0 \\
0 & 1-k & 0 & 0 \\
0 & 0 & \frac{2\zeta_{6}-1}{3}-k & \frac{-\zeta_{6}-1}{3} \\
0 & 0 & \frac{4\zeta_{6}-2}{3} & \frac{\zeta_{6}+1}{3}-k
\end{vmatrix}
&=\frac{(1-k)^{2}}{9} 
\begin{vmatrix} 
2\zeta_{6}-1-3k & -\zeta_{6}-1 \\
4\zeta_{6}-2 & \zeta_{6}+1-3k
\end{vmatrix}
\\
&=0,\end{align*}
hence $(1-k)^{2}(k^{2}-\zeta_{6}k+\zeta_{6}^{2})=0.$
Thus $k=1, \zeta_{3}$.

If $k=1$ hence 
\[ \begin{pmatrix}
1 & 0 & 0 & 0 \\
0 & 1 & 0 & 0 \\
0 & 0 & \frac{2\zeta_{6}-1}{3} & \frac{-\zeta_{6}-1}{3} \\
0 & 0 & \frac{4\zeta_{6}-2}{3} & \frac{\zeta_{6}+1}{3}
\end{pmatrix}
\begin{pmatrix} X \\ Y \\ Z \\ W \end{pmatrix}
=\begin{pmatrix} X \\ Y \\ Z \\ W \end{pmatrix}
\]
then we have $2Z+\zeta_{6}W=0$.
Therefore we have a fixed curve 
$\{[X:Y:Z:W] \in S \mid 2Z+\zeta_{6}W=0 \}=\{F_{4}(X,Y)-\frac{9\zeta_{6}}{16}W^{4}=0\}$
of genus 3.

If $k=\zeta_{3}$ hence 
\[ \begin{pmatrix}
1 & 0 & 0 & 0 \\
0 & 1 & 0 & 0 \\
0 & 0 & \frac{2\zeta_{6}-1}{3} & \frac{-\zeta_{6}-1}{3} \\
0 & 0 & \frac{4\zeta_{6}-2}{3} & \frac{\zeta_{6}+1}{3}
\end{pmatrix}
\begin{pmatrix} X \\ Y \\ Z \\ W \end{pmatrix}
=\zeta_{3}\begin{pmatrix} X \\ Y \\ Z \\ W \end{pmatrix}
\]
then we have $X=Y=0, Z=\zeta_{6}, W=1$.
Therefore we have an isolated point $[X:Y:Z:W] =[0:0:\zeta_{6}:1]$.
These imply that $S^{G_{P}}$ consists of one non-singular curve of genus 3 and one isolated point,
hence $(S, G_{P})$ is an Eisenstein $K3$ surface of type $(4, 3)$.

Similarly in other cases, we can check that the fixed locus consists of 
one non-singular curve of genus 3 and one isolated point by the same argument.
\end{proof}

\begin{prop}\label{1to1}
An Eisenstein $K3$ surface $(S, G)$ of type $(4, 3)$
gives a smooth quartic surface with an inner Galois point.
\end{prop}
\begin{proof}
There exists an isomorphism $\phi: S \to \mathbb{P}^{3}$
associated to $|C^{(3)}|$, and 
$\phi(S)$ is a smooth quartic surface given by 
$F_{4}(X, Y, Z)+F_{1}(X, Y, Z)W^{3}=0$ by Remark \ref{rem_AS49}.
In particular $[0:0:0:1]$ is a Galois point of $\phi (S)$.
\end{proof}

\begin{rem}
Since $S$ is not necessarily given as a hypersurface in $\mathbb{P}^{3}$,
$S$ does not necessarily have a Galois point.
The isomorphism $\phi$ is called a Galois embedding. 
See also \cite{Yoshihara-embedding}.
\end{rem}

\begin{rem}
It is known that some Eisenstein $K3$ surfaces correspond to some 
log del Pezzo surfaces of index 3 \cite{OT}.
Indeed an inner Galois point of a smooth quartic surface corresponds to 
a non-Du Val quotient (log terminal) singularity 
of log del Pezzo surfaces of index 3 with multiple smooth divisor property.
\end{rem}

\section{The quartic surface with the maximal number of inner Galois points.}\label{s8hax3}
Let $S_{8}$ be the quartic surface with the maximal number of inner Galois points,
hence it is given by $XY^{3}+X^{4}+ZW^{3}+Z^{4}=0$ by Proposition \ref{equation_place} (4).
Let $E_{1}$ and $E_{2}$ be elliptic curves defined by affine equations
$y^{2}=(Y/X)^{3}+1$ and $y^{2}=(W/Z)^{3}+1$,  respectively.
Note that $E_{1}$ and $E_{2}$ are mutually isogenous with complex multiplications.

\begin{lem}\label{inose-kummer}
The followings hold.
\begin{enumerate}
\item $S_{8}$ is singular, that is, its Picard number is 20.

\item Put $\sigma: =(\sigma_{1}\circ \sigma_{3})^{3}$, hence
$\sigma$ is a projective transformation given by 
$[X: Y: Z: W] \mapsto [X: Y: -Z: -W]$.
Then the quotient surface $S_{8}/\langle \sigma \rangle$ is 
birationally isomorphic to the Kummer surface $\Km (E_{1}\times E_{2})$
associated with the product of elliptic curves $E_{1}$ and $E_{2}$.
\end{enumerate}
\end{lem}
\begin{proof}
These follow from \cite[Theorem 2]{Inose}.
\end{proof}
\begin{rem}
The 8 singularities of  $S_{8}/\langle \sigma \rangle$
are correspond to the 8 Galois points.
See also \cite[p. 94]{KTY}.
\end{rem}

The transcendental lattice of $S_{8}$ is the orthogonal complement of 
the N\'{e}ron-Severi lattice of $S_{8}$ in $H^{2}(S_{8}, \mathbb{Z})$.
Hence the rank of the transcendental lattice of $S_{8}$ is 2.

\begin{prop}(\cite[Theorem 4]{SI})
Let $\mathcal{Q}$ be the set of $2\times 2$ positive definite even integral 
matrices:
\[ \mathcal{Q}=\left\{ \begin{pmatrix} 2a & b \\ b & 2c \end{pmatrix} \mid 
a,b,c \in \mathbb{Z}, a, c >0, b^{2}-4ac<0 \right\}.\]
There exists a bijective correspondence from the set of isomorphism classes of
singular $K3$ surfaces onto $\mathcal{Q}/SL_{2}(\mathbb{Z})$
which given by a singular $K3$ surface maps to  the Gram matrix of its transcendental lattice.
\end{prop}

The above proposition claims that 
it is important to study the transcendental lattice of a singular $K3$ surface.

\begin{thm}\label{mt2}
The Gram matrix of the transcendental lattice of $S_{8}$ is 
$\begin{pmatrix}8 & 4\\ 4 & 8\end{pmatrix}$.
\end{thm}
\begin{proof}
We remark  that the the Gram matrix of the transcendental lattice of $S_{8}$
is given by 4 times the Gram matrix of the transcendental lattice of $E_{1}\times E_{2}$
by \cite[Remark 2]{Kummer-Ni}, \cite[Proposition 4.3 (i)]{Morrison}
and \cite[Theorem 2]{Inose}.
Thus we may put the Gram matrix of the transcendental lattice of $E_{1}\times E_{2}$
is $\begin{pmatrix} 2a & b \\ b & 2c \end{pmatrix}$ and 
the Gram matrix of the transcendental lattice of $S_{8}$
is $\begin{pmatrix} 8a & 4b \\ 4b & 8c \end{pmatrix}$.

Since elliptic curves $E_{1}$ and $E_{2}$ are given by
$E_{1}=\mathbb{C}/\mathbb{Z}+\zeta_{3}\mathbb{Z}$ and 
$E_{2}=\mathbb{C}/\mathbb{Z}+\zeta_{3}\mathbb{Z}\simeq \mathbb{C}/\mathbb{Z}+(-\zeta_{3}^{2})\mathbb{Z}$,
we have
\[ 
\zeta_{3}=\frac{-b+\sqrt{b^{2}-4ac}}{2a}, \ \ \  -\zeta_{3}^{2}=\frac{b+\sqrt{b^{2}-4ac}}{2}\ \ 
\]
by \cite[$\S 3$ (I)]{SM}.
Thus we have $a=b=c=1$.
\end{proof}


\begin{thebibliography}{10}
\bibitem{AS}
M.~Artebani, A.~Sarti, 
Non-symplectic automorphisms of order 3 on $K3$ surfaces, 
Math. Ann. \textbf{342} (2008), 903--921.


\bibitem{Fukasawa}S.~Fukasawa,
Galois points for a plane curve in arbitrary characteristic,
Geom Dedicata \textbf{139} (2009), 211--218.

\bibitem{Homma}
M.~Homma,
Galois points for a Hermitian curve,
Commun. Algebra \textbf{34} (2006), 4503--4511.

\bibitem{Huybrechts}
D.~Huybrechts,
\textit{Lectures on K3 surfaces},
Cambridge Studies in Advanced Mathematics, \textbf{158}. 
Cambridge University Press, Cambridge, 2016. 

\bibitem{Inose}
H.~Inose, 
On certain Kummer surfaces which can be realized as non-singular quartic surfaces in $\mathbf{P}^{3}$,
J. Fac. Sci. Univ. Tokyo Sect. IA Math. \textbf{23} (1976), 545--560.

\bibitem{KTY}
M.~Kanazawa, T.~Takahashi, H.~Yoshihara,
The group generated by automorphisms belonging to Galois points of the quartic surface,
Nihonkai Math. J. \textbf{12} (2001), no.1, 89--99.


\bibitem{Kondo-book}
S.~Kondo, 
\textit{K3 surfaces},
EMS Tracts in Mathematics, \textbf{32},
European Mathematical Society Publishing House, 2020.


\bibitem{MOT}
S.~Ma, H.~Ohashi and S.~Taki, Rationality of the moduli spaces of Eisenstein $K3$ surfaces, 
Trans. Amer. Math. Soc. \textbf{367} (2015), no. 12, 8643--8679.

\bibitem{MY}
K.~Miura, H.~Yoshihara,
Field theory for function fields of plane quartic curves,
J. Algebra \textbf{226} (2000), 283--294.

\bibitem{M}
K.~Miura,
Field theory for function fields of singular plane quartic curves,
Bull. Aust. Math. Soc. \textbf{62} (2000), 193--204.


\bibitem{Morrison} 
D.R.~Morrison, 
On $K3$  surfaces with large Picard number, 
Invent. Math. \textbf{75} (1984), no. 1, 105--121.

\bibitem{Kummer-Ni}
V.V.~Nikulin, 
On Kummer surfaces (Russian) , 
Izv. Akad. Nauk SSSR Ser. Mat. \textbf{39} (1975), no. 2, 278--293.

\bibitem{Ni}
V.V.~Nikulin, 
Finite automorphism groups of K\"ahlerian $K3$ surfaces, 
Trans. Moscow Math. Soc. \textbf{38} (1980), No 2, 71--135.

\bibitem{OT}
H.~Ohashi and S.~Taki, 
$K3$ surfaces and log del Pezzo surfaces of index three, 
Manuscripta Math. \textbf{139} (2012), no. 3--4, 443--471.

\bibitem{SI}
T.~Shioda, H.~Inose, 
On singular $K3$ surfaces, 
Complex analysis and algebraic geometry,
119--136. Iwanami Shoten, Tokyo, 1977.

\bibitem{SM}
T.~Shioda, N.~Mitani, 
Singular abelian surfaces and binary quadratic forms, 
in \textit{Classification of Algebraic Varieties and Compact Complex Manifolds},
p. 259--287, Lecture Notes in Mathematics, vol 412. 
Springer, Berlin, Heidelberg, 1974.

\bibitem{Taki1}
S.~Taki, 
Classification of non-symplectic automorphisms of order 3 on $K3$ surfaces,
Math. Nachr. \textbf{284} (2011), 124--135.

\bibitem{Yoshihara1}
H.~Yoshihara, 
Function field theory of plane curves by dual curves,
J. Algebra \textbf{239} (2001), 340--355.

\bibitem{Yoshihara-GK3}
H.~Yoshihara, 
Galois points on quartic surfaces,
J. Math. Soc. Japan \textbf{53} (2001), 731--743.

\bibitem{Yoshihara2}
H.~Yoshihara,
Galois points for smooth hypersurfaces,
J. Algebra \textbf{264} (2003), 520--534.

\bibitem{Yoshihara-embedding}
H.~Yoshihara, 
Galois embedding of algebraic variety and its application to abelian surface,
Rend. Sem. Mat. Univ. Padova \textbf{117} (2007), 69--85.

\bibitem{open-problems}
H.~Yoshihara and S.~Fukasawa,
List of problems,
\url{https://sites.google.com/sci.kj.yamagata-u.ac.jp/fukasawa-lab/open-questions-english/}

\end{thebibliography}
\end{document}